\documentclass[12pt]{article}

\usepackage{empheq}
\usepackage{amsmath, amsthm}
\usepackage{amssymb}
\usepackage{amsfonts}
\usepackage{latexsym}
\usepackage{bm}
\numberwithin{equation}{section}

\usepackage{graphicx}
\usepackage{epsfig}
\usepackage{epstopdf}
\usepackage{subfigure}
\usepackage{overpic}
\usepackage{tikz}
\usepackage{float}
\usepackage{multirow}
\usepackage{tabularx}
\usepackage{booktabs}

\usepackage{color}

\usepackage[numbers, sort&compress]{natbib}
\usepackage[colorlinks=true]{hyperref}
\hypersetup{
    linkcolor = blue,
    citecolor = blue,
    urlcolor  = cyan,
    filecolor = blue,
}
\usepackage[capitalize, noabbrev]{cleveref}

\usepackage{indentfirst}
\usepackage{pifont}
\usepackage{titlesec}
\usepackage{enumitem}

\titleformat{\section}
  {\normalfont\Large\bfseries}{\thesection.}{1em}{}
\titleformat{\subsection}
  {\normalfont\large\bfseries}{\thesubsection.}{1em}{}
\titleformat{\subsubsection}
  {\normalfont\normalsize\bfseries}{\thesubsubsection.}{1em}{}

\newlist{assumptions}{enumerate}{1}
\setlist[assumptions]{
    label      = (\textit{A}\textsubscript{\arabic*}),
    ref        = (\textit{A}\textsubscript{\arabic*}),
    leftmargin = *,
    align      = left,
    itemindent = 0pt,
}

\newtheorem{theorem}{Theorem}[section]

\newtheorem{definition}[theorem]{Definition}
\newtheorem{example}[theorem]{Example}
\newtheorem{lemma}[theorem]{Lemma}
\newtheorem{proposition}[theorem]{Proposition}

\newtheoremstyle{problemstyle}
  {6pt}{6pt}
  {\normalfont}{}
  {\scshape}{.}
  { }
  {\thmname{#1}\thmnote{ (#3)}}

\theoremstyle{problemstyle}
\newtheorem*{problem}{Problem}

\renewenvironment{proof}[1][Proof]{\textbf{#1.} }
  {\ \rule{0.75em}{0.75em}\smallskip}

\begin{document}

\begin{center}
    {\Large\bfseries A Finite Element Method for Elliptic Hemivariational Inequalities
    in Non-isotropic and Heterogeneous Semipermeable Media}
\end{center}

\begin{center}
    \large {\sc Ban Li}\footnote{College of Mathematics and Systems Science,
        Xinjiang University, Urumqi, Xinjiang 830017, PR China.
        Email: liban1228@163.com},
    \quad
    \large {\sc Bangmin Wu}\footnote{College of Mathematics and Systems Science,
        Xinjiang University, Urumqi, Xinjiang 830017, PR China. The work of this author was partially supported by the Natural Science Foundation of Xinjiang Uygur (Grant No.~2024D01C227) and the Science and Technology Program of Xinjiang Uyghur (Grant No.~2025A03011-3).
        \textsuperscript{*} Corresponding author: bmwu\_math@xju.edu.cn}\textsuperscript{,*}
\end{center}

\begin{quote}
    {\bf Abstract.}
    This work investigates finite element approximations for a general class of elliptic hemivariational inequalities arising in semipermeable media. The proposed model incorporates non-isotropic and heterogeneous diffusion coefficients, alongside both interior and boundary semipermeability terms, extending the isotropic and homogeneous framework examined by Han (2019). The existence and uniqueness of solutions are rigorously established. An optimal a priori error estimate for the linear finite element approximation is derived under appropriate solution regularity assumptions. Numerical experiments are presented to corroborate the theoretical analysis and to confirm the optimal convergence rates for the non-isotropic and heterogeneous case.

    {\bf Keywords.}  Hemivariational inequalities, Non-isotropic and heterogeneous semipermeable media, Finite element method, Error estimates, Numerical simulation.

    {\bf AMS Classification.} 65N30, 49J40
\end{quote}

\section{Introduction}

Problems involving semipermeable media often arise in the context of porous media flow, where semipermeability may occur both in the interior and on the boundary. Such problems were initially investigated in \cite{Duvaut1972Les} under a monotone semipermeable relation, leading to formulations as variational inequalities. However, due to their reliance on monotonicity and convexity, variational inequalities are limited in capturing nonsmooth and nonconvex physical processes. To overcome this limitation, Panagiotopoulos \cite{Panagiotopoulos983nonconvex} introduced hemivariational inequalities in 1983, based on Clarke’s generalized gradient for locally Lipschitz functionals \cite{Clarke1975gen}. This framework was subsequently applied by Panagiotopoulos \cite{Panagiotopoulos1985semipermeable} to extend semipermeable media problems to nonmonotone relations. Because hemivariational inequalities can rigorously describe the behavior of energy functionals at non-differentiable points, their mathematical theory and numerical methods have since developed rapidly \cite{migorski2013non, Sofonea2018var}.

The numerical analysis of hemivariational inequalities has emerged as a significant area of research in recent years. For the finite element discretization of such problems, \cite{Haslinger1999fem} offers a comprehensive treatment and introduces several numerical algorithms. The work in \cite{Han2014Aclass} rigorously establishes the convergence of the linear finite element method for hemivariational inequalities and derives optimal-order error estimates for the numerical solution. This has been followed by a series of subsequent studies dedicated to obtaining optimal-order error estimates for linear finite element approximations applied to a diverse range of hemivariational inequalities \cite{Han2019contact, Han2017elliptic}.

For the semipermeable media problem, Han \cite{han2019media} established the well-posedness and derived optimal finite element error estimates for the isotropic and homogeneous case. While the extension to non-isotropic and heterogeneous media was noted, a detailed analysis and corresponding numerical simulations were not provided, which motivates the present study. Subsequently, Han \cite{han2021numericalparabolic} developed a fully discrete numerical scheme with optimal error estimates for parabolic hemivariational inequalities in semipermeable media. Wang \cite{wang2020discontinuous} introduced a discontinuous Galerkin method for an elliptic hemivariational inequality in such media and established a priori error estimates. In summary, although the theoretical and numerical analysis of hemivariational inequalities for semipermeable media has advanced considerably in recent years, the majority of existing work remains focused on isotropic and homogeneous settings.

However, real materials are intrinsically heterogeneous. Consequently, extending the analysis to incorporate variable coefficients is crucial for accurate physical modeling. The inclusion of non-isotropic and heterogeneous coefficients necessitates a corresponding generalization of the well-posedness and error analysis. The principal contribution of this work is to generalize the finite element analysis presented in \cite{han2019media} to a non-isotropic, heterogeneous elliptic operator. Specifically, we establish the well-posedness of the resulting hemivariational inequality and derive an optimal first-order error estimate for its linear finite element discretization. Numerical experiments are provided to confirm the theoretical convergence rate and to demonstrate the efficacy of the proposed method for simulating non-isotropic and heterogeneous media.

This paper is organized as follows. In Section \ref{sec:Hemivariational inequality}, essential preliminary material is introduced, the model problem is formulated, and the corresponding hemivariational inequality is derived. Section \ref{sec:AnalysisHemivariational Inequality} is devoted to establishing the existence and uniqueness of solutions for the non-isotropic and heterogeneous semipermeable media problem. In Section \ref{sec:fem}, a linear finite element discretization is proposed, the uniform boundedness of the discrete solutions is proven, and an optimal error estimate in the $H^1$ norm is established. Finally, Section \ref{sec:numerical example} presents numerical experiments that validate the derived theoretical error estimates.

\section{Model problem for non-isotropic and heterogeneous semipermeable media}\label{sec:Hemivariational inequality}

\subsection{Preliminaries}
Throughout, all linear spaces are real and we use the following notation. Let $X$ be a normed space with norm $\|\cdot\|_X$. The duality pairing between $X^*$ (the topological dual of $X$) and $X$ is denoted by $\langle\cdot,\cdot\rangle$. Let $0_X$ be the zero element of $X$ and $2^{X^*}$ the set of all subsets of $X^*$. Unless stated otherwise, $X$ is a Banach space. For two normed spaces $X$ and $W$, $\mathcal{L}(X,W)$ denotes the space of all continuous linear operators from $X$ to $W$.

\begin{definition}
	Let $\psi: X \to \mathbb{R}$ be a locally Lipschitz functional. Its generalized (Clarke) directional derivative at $x\in X$ in the direction $v\in X$ is defined by
	\begin{equation}\nonumber\label{phi0}
		\psi^0(x;v) := \limsup_{y\to x,\ \lambda\downarrow 0} \frac{\psi(y+\lambda v) - \psi(y)}{\lambda}.
	\end{equation}
	The generalized gradient (subdifferential) of $\psi$ at $x$ is
	\begin{equation}\nonumber\label{partialpsi0}
		\partial\psi(x) := \{\,\zeta\in X^* \mid \psi^0(x;v) \ge \langle\zeta,v\rangle\ \ \forall v\in X\,\}.
	\end{equation}
\end{definition}

We refer to \cite{clarke1990optim,naniewicz1995Hemivar} for properties of the Clarke subdifferential and generalized directional derivative. The latter is subadditive:
\begin{proposition}\label{pro:Jsub}
	Let $\psi:X\to\mathbb{R}$ be locally Lipschitz. For every $u\in X$, the map $v\mapsto \psi^0(u;v)$ is subadditive, i.e.,
	\begin{equation}\nonumber\label{subadd}
		\psi^0(u;v_1+v_2) \le \psi^0(u;v_1) + \psi^0(u;v_2) \qquad \forall\,v_1,v_2\in X.
	\end{equation}
\end{proposition}

\subsection{Model problem}
For non-isotropic and heterogeneous semipermeable media, both boundary and interior semipermeability conditions are typically involved.
Let $\Omega\subset\mathbb{R}^d$ be an open, bounded, connected domain with a Lipschitz boundary $\partial\Omega$, partitioned into two disjoint parts $\Gamma_S$ and $\Gamma_D$ such that $\operatorname{meas}(\Gamma_D)>0$.
Consider the elliptic operator
\begin{equation}\label{Lu}
	Lu = -\sum_{i,j=1}^d \frac{\partial}{\partial x_j}\Bigl(a_{ij}\frac{\partial u}{\partial x_i}\Bigr) + a_0 u,
\end{equation}
whose coefficients satisfy \cite[Subsection 8.4.5]{atkin2009theo}
\begin{align}
	a_{ij}, a_0 &\in L^\infty(\Omega), \label{eq:aija0} \\
	\sum_{i,j=1}^d a_{ij}\xi_i\xi_j &\ge \theta\,\lvert\bm{\xi}\rvert^2 \quad \forall\bm{\xi}\in\mathbb{R}^d,\;\text{a.e. in }\Omega, \label{eq:xi} \\
	a_0 &\ge 0 \quad \text{a.e. in }\Omega, \label{eq:a0}
\end{align}
with a constant $\theta>0$.
The co-normal derivative on the boundary is defined as
\begin{equation}\nonumber\label{nonnumann}
	\frac{\partial u}{\partial\nu_L} = \sum_{i,j=1}^d a_{ij}\,\frac{\partial u}{\partial x_i}\,\nu_j .
\end{equation}
Let $f_0\in L^2(\Omega)$ and let the interior semipermeability be encoded by the subdifferential relation $-f_1\in\partial j_1(u)$.
The pointwise formulation of the model problem then reads
\begin{align}
	Lu &= f_0 + f_1 && \text{in } \Omega, \label{model} \\
	u &= 0 && \text{on } \Gamma_D, \label{drichlet} \\
	-\frac{\partial u}{\partial \nu_L} &\in \partial j_2(u) \qquad && \text{on } \Gamma_S. \label{neumann}
\end{align}
Here $\partial j_1$ and $\partial j_2$ denote the Clarke subdifferentials of locally Lipschitz functionals $j_1$ and $j_2$, representing the interior and boundary semipermeability conditions, respectively.

The functions $j_i$ ($i=1,2$) are assumed to satisfy the following conditions.

\begin{subequations}\label{eq:main}
	\textbf{H($j$).}\quad
	\begin{empheq}[left=\empheqlbrace]{align}
		& j_1 \text{ and } j_2 \text{ are locally Lipschitz continuous;} \label{eq:cond:a} \\
		& \text{there exist constants } \bar{c}_{i,0},\bar{c}_{i,1} \text{ such that} \nonumber \\
		&\qquad |\partial j_i(t)| \le \bar{c}_{i,0} + \bar{c}_{i,1}|t| \quad \forall\,t\in\mathbb{R},\ i=1,2; \label{eq:cond:b} \\
		& \text{there exist constants } \alpha_i \text{ such that} \nonumber \\
		&\qquad j_i^0(t_1;t_2-t_1) + j_i^0(t_2;t_1-t_2) \le \alpha_i|t_1-t_2|^2 \quad \forall\,t_1,t_2\in\mathbb{R},\ i=1,2. \label{eq:cond:c}
	\end{empheq}
\end{subequations}

\subsection{Hemivariational inequality}
The weak formulation of (\ref{model})--(\ref{neumann}) is posed in the space
\begin{equation}\nonumber\label{Vspace}
	V = \{ v \in H^1(\Omega) \mid v = 0 \text{ on } \Gamma_D \},
\end{equation}
endowed with the norm \(\|v\|_{V} = \|\nabla v\|_{L^2(\Omega)}\), which is equivalent to the standard \(H^1\)-norm on \(V\) by the Poincaré inequality. The corresponding bilinear form is
\begin{equation}\nonumber\label{auv}
	a(u,v) = \int_\Omega \Bigl( \sum_{i,j=1}^d a_{ij} \frac{\partial u}{\partial x_i} \frac{\partial v}{\partial x_j} + a_0 uv \Bigr) \, dx, \qquad u,v\in V.
\end{equation}
It can be shown that \(a(\cdot,\cdot)\) is bounded and coercive on \(V\). Indeed, by the Cauchy–Schwarz inequality, the boundedness of the coefficients and the Poincaré inequality we obtain,
\begin{equation}\label{bounded}
	\begin{aligned}
		|a(u,v)|
		&\le \sum_{i,j=1}^d \|a_{ij}\|_{L^\infty(\Omega)} \Big\| \frac{\partial u}{\partial x_i} \Big\|_{L^2(\Omega)} \Big\| \frac{\partial v}{\partial x_j} \Big\|_{L^2(\Omega)}
		+ \|a_0\|_{L^\infty(\Omega)} \|u\|_{L^2(\Omega)} \|v\|_{L^2(\Omega)} \\
		&\le c_a \|u\|_V \|v\|_V,
	\end{aligned}
\end{equation}
where \(c_a\) depends on \(\|a_{ij}\|_{L^\infty(\Omega)}\) and \(\|a_0\|_{L^\infty(\Omega)}\). Using \eqref{eq:xi} and \eqref{eq:a0}, the coercivity follows:
\begin{equation}\label{coercive}
	a(v,v) \ge \theta \|\nabla v\|_{L^2(\Omega)}^2 = \theta \|v\|_V^2.
\end{equation}

Multiplying \eqref{model} by a test function \(v\in V\) and applying Green's formula yields
\begin{equation}\nonumber
	a(u,v) + \int_{\Gamma_S} \Bigl( -\frac{\partial u}{\partial\nu_L} \Bigr) v \, ds + \int_\Omega (-f_1) v \, dx = \int_\Omega f_0 v \, dx \qquad \forall\, v\in V.
\end{equation}
Because \(-f_1 \in \partial j_1(u)\) and \(-\frac{\partial u}{\partial\nu_L} \in \partial j_2(u)\), the definition of the Clarke subdifferential gives the pointwise bounds
\((-f_1)v \le j_1^0(u;v)\) a.e. in \(\Omega\) and \(-\frac{\partial u}{\partial\nu_L} v \le j_2^0(u;v)\) on \(\Gamma_S\). Inserting these into the weak formulation leads to the following hemivariational inequality.
\begin{problem}[\(P_m\)]\label{modelproblemhvi}
	Find \(u\in V\) such that
	\begin{equation}\label{HVI1}
		a(u,v) + \int_\Omega j_1^0(u;v) \, dx + \int_{\Gamma_S} j_2^0(u;v) \, ds \ge \int_\Omega f_0 v \, dx \qquad \forall\, v\in V.
	\end{equation}
\end{problem}

\section{Analysis of the hemivariational inequality}\label{sec:AnalysisHemivariational Inequality}

To establish the well-posedness of the non-isotropic and heterogeneous semipermeable media problem $(P_m)$, we invoke the following abstract result, which specializes \cite[Theorem 3.1]{han2019media} with $K=X$.
\begin{theorem}\label{th:generalth}
	Let $X$ be a reflexive Banach space and assume:
	\begin{assumptions}
		\item \label{ass:A1} For $i=1,2$, $X_i$ is a Banach space and $\gamma_i\in\mathcal{L}(X,X_i)$ satisfies, for some $c_i>0$,
		\begin{equation}\nonumber\label{gammai}
			\|\gamma_i v\|_{X_i}\le c_i\|v\|_X\qquad\forall v\in X.
		\end{equation}
		\item \label{ass:A2} $A:X\to X^*$ is bounded, continuous and strongly monotone: for some $m_A>0$,
		\begin{equation}\nonumber\label{Amono}
			\langle Av_1-Av_2,v_1-v_2\rangle\ge m_A\|v_1-v_2\|_X^2\qquad\forall v_1,v_2\in X.
		\end{equation}
		\item \label{ass:A3} $J_i:X_i\to\mathbb{R}$ is locally Lipschitz and there exist constants $c_{i,0},c_{i,1},\alpha_i\ge0$ such that
		\begin{align}
			\|\partial J_i(z)\|_{X_i^*} &\le c_{i,0}+c_{i,1}\|z\|_{X_i} \quad\forall z\in X_i,\label{ass:A3a}\\
			J_i^0(z_1;z_2-z_1)+J_i^0(z_2;z_1-z_2) &\le \alpha_i\|z_1-z_2\|_{X_i}^2 \quad\forall z_1,z_2\in X_i.\label{ass:A3b}
		\end{align}
		\item \label{ass:A4} $\displaystyle \alpha_1c_1^2+\alpha_2c_2^2 < m_A$.
		\item \label{ass:A5} $f\in X^*$.
	\end{assumptions}
	Then there exists a unique $u\in X$ solving
	\begin{equation}\label{generalhvi}
		\langle Au,v-u\rangle + J_1^0(\gamma_1u;\gamma_1v-\gamma_1u) + J_2^0(\gamma_2u;\gamma_2v-\gamma_2u) \ge \langle f,v-u\rangle \qquad\forall v\in X.
	\end{equation}
\end{theorem}

We now prove the existence and uniqueness of a solution to Problem $(P_m)$.
Set $V_1 = L^2(\Omega)$, $V_2 = L^2(\Gamma_S)$ and let $\gamma_1: V \to V_1$ be the embedding operator, $\gamma_2: V \to V_2$ the trace operator.
Define $A: V \to V^*$ by
\begin{equation}\nonumber\label{Auv}
	\langle Au,v\rangle = a(u,v)\qquad\forall\, u,v\in V.
\end{equation}
Let $\lambda_L$ be the smallest eigenvalue of the problem
\[
\begin{aligned}
	Lu &= \lambda u \quad \text{in } \Omega,\\
	u  &= 0 \qquad\text{on } \Gamma_D,\\
	\frac{\partial u}{\partial\nu_L} &= 0 \qquad\text{on } \Gamma_S .
\end{aligned}
\]
Testing the first equation with $u$, integrating by parts, and using the boundary conditions togethe we obtain
\[
a(u,u) = \lambda \|u\|_{L^2(\Omega)}^2 .
\]
Hence, by the definition of the smallest eigenvalue,
\begin{equation}\label{gamma1u}
	\|\gamma_1 u\|_{L^2(\Omega)} = \|u\|_{L^2(\Omega)} \le \lambda_L^{-\frac12} \sqrt{a(u,u)} \le \lambda_L^{-\frac12} \|u\|_V .
\end{equation}
Similarly, let $\mu_L$ be the smallest eigenvalue of
\[
\begin{aligned}
	Lu &= 0 \qquad\text{in } \Omega,\\
	u  &= 0 \qquad\text{on } \Gamma_D,\\
	\frac{\partial u}{\partial\nu_L} &= \mu u \quad \text{on } \Gamma_S .
\end{aligned}
\]
Then, by the Poincaré inequality and the trace theorem,
\begin{equation}\label{gamma2u}
	\|\gamma_2 u\|_{L^2(\Gamma_S)} = \|u\|_{L^2(\Gamma_S)} \le \mu_L^{-\frac12} \sqrt{a(u,u)} \le \mu_L^{-\frac12} \|u\|_V .
\end{equation}
Therefore, Assumption~\ref{ass:A1} is verified with $c_1 = \lambda_L^{-\frac12}$, $c_2 = \mu_L^{-\frac12}$.

For Assumption~\ref{ass:A2}, set $w = v_1 - v_2$; then
\begin{equation}\label{modelA2}
	\begin{aligned}
		\langle A v_1 - A v_2, w \rangle &= a(v_1,w) - a(v_2,w) \\
		&= a(w,w) \\
		&\ge \theta \|w\|_V^2 .
	\end{aligned}
\end{equation}

Define the two functionals
\begin{equation}\label{J1}
	J_1: V_1 \to \mathbb{R},\quad J_1(v) = \int_\Omega j_1(v)\,dx,
\end{equation}
\begin{equation}\label{J2}
	J_2: V_2 \to \mathbb{R},\quad J_2(v) = \int_{\Gamma_S} j_2(v)\,ds .
\end{equation}

From \cite[Theorem 3.47]{migorski2013non} we have
\begin{equation}\label{J10}
	J_1^0(v;w) \le \int_\Omega j_1^0(v;w)\,dx \qquad \forall\, v,w\in V_1,
\end{equation}
\begin{equation}\label{J20}
	J_2^0(v;w) \le \int_{\Gamma_S} j_2^0(v;w)\,ds \qquad \forall\, v,w\in V_2 .
\end{equation}
Using \eqref{J1} and \eqref{eq:cond:b},
\begin{equation}
	\begin{aligned}
		\|\partial J_1(z)\| &\le |J_1^0(z,w)| \\
		&\le \Bigl| \int_\Omega j_1^0(z,w)\,dx \Bigr| \\
		&\le \Bigl| \int_\Omega \bigl( \bar c_{1,0} + \bar c_{1,1}|z| \bigr) dx \Bigr|,
	\end{aligned}
\end{equation}
further
\begin{equation}\label{partialJz1}
	\|\partial J_1(z)\|_{L^2(\Omega)} \le c_{1,0} + c_{1,1} \|z\|_{L^2(\Omega)},
\end{equation}
and similarly,
\begin{equation}\label{partialJz2}
	\|\partial J_2(z)\|_{L^2(\Gamma_S)} \le c_{2,0} + c_{2,1} \|z\|_{L^2(\Gamma_S)} .
\end{equation}
Here the constants $c_{1,0},c_{1,1}$ depend on $\bar c_{1,0},\bar c_{1,1}$ and $|\Omega|$, while $c_{2,0},c_{2,1}$ depend on $\bar c_{2,0},\bar c_{2,1}$ and $|\Gamma_S|$.

From \eqref{J10} and \eqref{eq:cond:c} we obtain
\begin{equation}\label{jioapl}
	\begin{aligned}
		J_1^0(z_1;z_2-z_1) + J_1^0(z_2;z_1-z_2)
		&\le \int_{\Omega} \bigl( j_1^0(z_1;z_2-z_1) + j_1^0(z_2;z_1-z_2) \bigr) \, dx \\
		&\le \alpha_1 \| z_1 - z_2 \|_{L^2(\Omega)}^2,
	\end{aligned}
\end{equation}
and similarly,
\begin{equation}\label{j1j2apl}
	J_2^0(z_1;z_2-z_1) + J_2^0(z_2;z_1-z_2) \le \alpha_2 \| z_1 - z_2 \|_{L^2(\Gamma_S)}^2 .
\end{equation}
Thus Assumption~\ref{ass:A3} holds.  In the present setting, the smallness condition \ref{ass:A4} becomes
\begin{equation}\label{assa4}
	\alpha_1 \lambda_L^{-1} + \alpha_2 \mu_L^{-1} < \theta .
\end{equation}
To verify Assumption~\ref{ass:A5}, note that $f_0\in L^2(\Omega)$ gives
\begin{equation}
	\Bigl| \int_{\Omega} f_0 v \, dx \Bigr| \le \|f_0\|_{L^2(\Omega)} \|v\|_{L^2(\Omega)} \le \|f_0\|_{L^2(\Omega)} \|v\|_V .
\end{equation}
Applying \cref{th:generalth}, we conclude that under condition \eqref{assa4} there exists a unique solution $u\in V$ to the problem
\begin{equation}\label{Jhvi}
	u\in V,\quad \langle Au,v\rangle + J_1^0(\gamma_1 u;\gamma_1 v) + J_2^0(\gamma_2 u;\gamma_2 v) \ge \int_{\Omega} f_0 v \, dx \quad \forall\, v\in V .
\end{equation}
Finally, by \eqref{J10} and \eqref{J20}, any solution of \eqref{Jhvi} also solves Problem ($P_m$).

\begin{theorem}\label{th:uniqueu}
	Let $\Omega\subset\mathbb{R}^d$ be an open, bounded, connected domain.
	Assume that $f_0\in L^2(\Omega)$, Assumption $H(j)$ holds, and the smallness condition~\eqref{assa4} is satisfied.
	Then Problem ($P_m$) admits a unique solution $u\in V$.
\end{theorem}

\begin{proof}
	Existence of a solution has already been established; it remains to prove uniqueness.
	Let $u_1,u_2\in V$ be two solutions of Problem $(P_m)$.  They satisfy
	\[
	a(u_1,u_2-u_1) + \int_{\Omega} j_1^0(u_1;u_2-u_1)\,dx + \int_{\Gamma_{S}} j_2^0(u_1;u_2-u_1)\,ds \ge \int_{\Omega} f_0(u_2-u_1)\,dx,
	\]
	\[
	a(u_2,u_1-u_2) + \int_{\Omega} j_1^0(u_2;u_1-u_2)\,dx + \int_{\Gamma_{S}} j_2^0(u_2;u_1-u_2)\,ds \ge \int_{\Omega} f_0(u_1-u_2)\,dx .
	\]
	Adding these two inequalities and employing \eqref{gamma1u}, \eqref{gamma2u}, \eqref{modelA2}, and \eqref{j1j2apl} yields
	\begin{align*}
		\theta \|u_2-u_1\|_{V}^2
		&\le a(u_2-u_1,u_2-u_1) \\
		&\le \int_{\Omega} \bigl( j_1^0(u_1;u_2-u_1) + j_1^0(u_2;u_1-u_2) \bigr) dx \\
		&\quad + \int_{\Gamma_{S}} \bigl( j_2^0(u_1;u_2-u_1) + j_2^0(u_2;u_1-u_2) \bigr) ds \\
		&\le (\alpha_1\lambda_L^{-1} + \alpha_2\mu_L^{-1}) \|u_2-u_1\|_{V}^2 .
	\end{align*}
	Rearranging gives
	\[
	\bigl( \theta - \alpha_1\lambda_L^{-1} - \alpha_2\mu_L^{-1} \bigr) \|u_2-u_1\|_{V}^2 \le 0 .
	\]
	By \eqref{assa4} the factor in parentheses is strictly positive; therefore $\|u_2-u_1\|_V = 0$, i.e., $u_1 = u_2$.
\end{proof}

\section{Finite element approximation}\label{sec:fem}

For simplicity, this paper focuses on the two-dimensional case. We still retain the previous assumptions. In this section, we focus on the finite element method for Problem ($P_m$). Throughout the paper, $c$ denotes a generic positive constant independent of $h$.

Let $\mathcal T_h$ be a family of quasi-uniform triangulations of $\bar{\Omega}$ into elements $T$ that are compatible with the partition of the boundary $\partial\Omega = \Gamma_D\cup\Gamma_{S},$ in the sense that if an edge $e \subset \partial T$ satisfies meas($e \cap \Gamma^i_{D/S} )>0$, then $e \in \Gamma^i_{D/S}$. We use the linear element space corresponding
to $\mathcal{T}^h$:
\begin{equation}\nonumber\label{Vh}
	V_h = \begin{Bmatrix}
		v_h\in C(\overline{\Omega})\mid v_h|_T\in\mathbb{P}_1(T),\ T\in\mathcal{T}^h,\ v_h=0\ \mathrm{on}\ \Gamma_D
	\end{Bmatrix},
\end{equation}
where $\mathbb{P}_1(T)$ denotes the space of polynomial functions with the total degree no more than 1 on $T$. The discrete scheme of the finite element method  for Problem ($P_m$) is: 
\begin{problem}[$P_d$]
Find $u_h \in V_h$ such that
\begin{equation}\label{discreteHVI}
	a{(u_{h},v_{h})}+\int_{\Omega}j_1^0(u_{h};v_{h})\mathrm{d}x+\int_{\Gamma_{S}}j_2^0(u_{h};v_{h})\mathrm{d}s
	\geq\int_{\Omega}f_0v_{h}\mathrm{d}x,\quad\forall\, v_{h}\in V_{h}.
\end{equation}
\end{problem}

Since \(V_h\subset V\), the boundedness and coercivity of \(a(\cdot,\cdot)\) on
\(V_h\) follow directly from \eqref{bounded} and \eqref{coercive}. More precisely,
\begin{equation}\label{boundedh}
	|a(u_h,v_h)|
	\le c_a\|u_h\|_V\|v_h\|_V,\quad \forall\, u_h,v_h\in V_h.
\end{equation}

\begin{equation}\label{coerciveh}
	a(v_h,v_h)\ge \theta\|v_h\|_V^2,\quad \forall\, v_h\in V_h.
\end{equation}

The existence of the discrete solution $u_h$ follows from the coercivity of $a(\cdot,\cdot)$ and the growth conditions on $j_1,j_2$, while uniqueness is guaranteed by the smallness condition \eqref{assa4}, exactly as in the continuous case (\cref{th:uniqueu}). The next lemma shows that the numerical solutions are uniformly bounded in $V_h$.

\begin{lemma}\label{uhbounded}
	The numerical solutions of Problem ($P_d$) are uniformly bounded in $V_h$, independently of $h$.
\end{lemma}

\begin{proof}
	Taking $v_h = -u_h \in V_h$ in \eqref{discreteHVI}, we have 
	$$a(u_{h},u_{h})\leq\int_{\Omega}^{}j_1^0(u_{h};-u_{h})\mathrm{d}x+\int_{\Gamma_{S}}^{}j_2^0(u_{h};-u_{h})\mathrm{d}s+\int_{\Omega}f_0u_{h}\mathrm{d}x.$$
	
	Applying Assumption \eqref{eq:cond:c} pointwise with $t_1 = u_h \in V_h$ and $t_2 = 0$ (for $i = 1,2$), we obtain
	$$j_i^0(u_h;-u_h)\leq\alpha_i|u_h|^2 - j_i^0(0;u_h),$$
	and, by \eqref{eq:cond:b},
	$$|j_i^0(0;u_h)|\leq(\bar{c}_{i,0}+\bar{c}_{i,1}|0|)|u_h|=\bar{c}_{i,0}|u_h|.$$
	Thus,
	$$j_i^0(u_h;-u_h)\leq\alpha_i|u_h|^2+\bar{c}_{i,0}|u_h|,$$
	and hence, 
	\begin{equation}\label{a1c1}
		\int_{\Omega}j_1^0(u_{h};-u_{h})\mathrm{d}x\leq\alpha_1\|u_{h}\|_{L^2\left(\Omega\right)}^2+c_{1,0}\|u_{h}\|_{ L^2\left(\Omega\right)},
	\end{equation}
	\begin{equation}\label{a2c2}
		\int_{\Gamma_{S}}j_2^0(u_{h};-u_{h})\mathrm{d}s\leq\alpha_2\|u_{h}\|_{L^2\left(\Gamma_{S}\right)}^2+c_{2,0}\|u_{h}\|_{ L^2\left(\Gamma_{S}\right) }.
	\end{equation}
	
	Adding the two inequalities and using the coercivity \eqref{coerciveh}, the operator estimates \eqref{gamma1u}--\eqref{gamma2u} for $\gamma_1$ and $\gamma_2$, together with \eqref{a1c1}--\eqref{a2c2}, we obtain:
	
	\begin{equation}\nonumber
		\begin{aligned}
			\theta\|u_{h}\|_{V}^2
			& \leq a(u_{h},u_{h}) \\
			& \leq\|f_0\|_{L^2(\Omega)}\|u_{h}\|_{L^2(\Omega)}+\alpha_1\|u_{h}\|_{L^2(\Omega)}^2+c_{1,0}\|u_{h}\|_{L^2(\Omega)}+ \\
			& \quad\alpha_2\|u_{h}\|_{L^2(\Gamma_{S})}^2+c_{2,0}\|u_{h}\|_{L^2(\Gamma_{S})} \\
			& \leq c\|u_{h}\|_{V}+(\alpha_1\lambda_L^{-1}+\alpha_2\mu_L^{-1})\|u_{h}\|_{V}^2.
		\end{aligned}
	\end{equation}
	Hence,
	$$(\theta-\alpha_1\lambda_L^{-1}-\alpha_2\mu_L^{-1})\|u_{h}\|_{V}^2\leq c\|u_{h}\|_{V}.$$
	Since the smallness condition \eqref{assa4} holds, $\|u_{h}\|_{V}$ is uniformly bounded.
	
\end{proof}

We now establish a C\'ea-type estimate, which will be used to derive the optimal error bound for the finite element approximation.
\begin{theorem}\label{lem:cea}
	Under the conditions of Theorem \ref{th:uniqueu}, let $u \in V$ and $u_h \in V_h$ be the solutions of Problem ($P_m$) and Problem ($P_d$), respectively. Assume 
\begin{equation}\label{uregularity}
	u\in H^2(\Omega),\qquad
	u|_{\Gamma_S^i}\in H^2(\Gamma_S^i),\quad 1\le i\le m.
\end{equation}
	For all $v_h\in V_h$, the following estimate holds:
	\begin{equation}\label{eq:cea}
		\|u-u_{h}\|_{V} \leq c(u)\left(\|u-v_{h}\|_{V}+\|u-v_{h}\|_{L^2(\Omega)}^{\frac12}+\|u-v_{h}\|_{L^2(\Gamma_{S})}^{\frac12}\right),
	\end{equation}
	where $c(u)$ is independent of $h$ and depends on $\|u\|_{H^{2}(\Omega)}$ and $\|u\|_{H^{2}(\Gamma_{S}^{i})}$,  $1\leq i\leq m$.
\end{theorem}

\begin{proof}
	By the coercivity of the bilinear form \eqref{coercive}, we have
	\begin{equation}\label{au-uh}
		\begin{aligned}
			\theta\|u-u_{h}\|_V^{2}&\leq a(u-u_h,u-u_h)\\
			&=a(u-u_h,u-v_h)+a(u-u_h,v_h-u_h)\\
			&=a(u-u_h,u-v_h)+a(u,v_h-u)\\
			&\quad +a(u,u-u_h)-a(u_h,v_h-u_h)
		\end{aligned}
	\end{equation}
	Taking $v=u_h-u\in V$ in \eqref{HVI1}, we obtain
	\begin{equation}\label{auu-uh}
		\begin{aligned}
			a(u,u-u_h)\leq\int_{\Omega}j_1^0(u;u_h-u)\mathrm{d}x+\int_{\Gamma_{S}}j_2^0(u;u_h-u)\mathrm{d}s-\int_{\Omega}f_0(u_h-u)\mathrm{d}x.
		\end{aligned}
	\end{equation}
	Taking $v_h-u_h \in V_h$ as the test function in \eqref{discreteHVI}, we obtain
	\begin{equation}\label{auhvh-uh}
		-a(u_h,v_h-u_h)\leq\int_\Omega j_1^0(u_h;v_h-u_h)\mathrm{d}x+\int_{\Gamma_{S}}j_2^0(u_h;v_h-u_h)\mathrm{d}s-\int_{\Omega}f_0(v_h-u_h)\mathrm{d}x.
	\end{equation}
	Substituting  \eqref{auu-uh} and \eqref{auhvh-uh} into \eqref{au-uh}, we obtain
	\begin{equation}\label{au-uh2}
		\begin{aligned}\theta\| u - u_h \|_V^2 
			&\leq a(u - u_h, u - v_h) + a(u, v_h - u) \\
			&+ \left[ \int_{\Omega} j_1^0 (u; u_h - u) + \int_{\Gamma_{S}} j_2^0 (u; u_h - u) - \int_{\Omega} f_0 (u_h - u)\mathrm{d}x \right] \\
			&+ \left[ \int_{\Omega} j_1^0 (u_h; v_h - u_h) + \int_{\Gamma_{S}} j_2^0 (u_h; v_h - u_h) - \int_{\Omega} f_0 (v_h - u_h)\mathrm{d}x \right].
		\end{aligned}
	\end{equation}
	Rearranging the above inequality and using \cref{pro:Jsub}, we obtain
	\begin{equation}\label{ari}
		\theta\|u-u_h\|_V^2\leq a(u-u_h,u-v_h)+R(v_h-u)+I(u_h,v_h),
	\end{equation}
	where 
	\begin{equation}\label{R(w)}
		R(w)=a(u,w)+\int_{\Omega}j_1^0(u;w)\mathrm{d}x+\int_{\Gamma_{S}}j_2^0(u;w)\mathrm{d}s-\int_{\Omega}f_0w\mathrm{d}x,
	\end{equation}
	\begin{equation}\label{Iuhvh}
		\begin{split}
			I(u_h,v_h) &= \int_\Omega j_1^0(u; u_h - v_h) \, \mathrm{d}x + \int_\Omega j_1^0(u_h; v_h - u_h) \, \mathrm{d}x \\
			&\quad + \int_{\Gamma_{S}} j_2^0(u; u_h - v_h) \, \mathrm{d}s + \int_{\Gamma_{S}} j_2^0(u_h; v_h - u_h) \, \mathrm{d}s.
		\end{split}
	\end{equation}
	We now estimate the terms on the right-hand side of \eqref{ari}. By the boundedness of $a(\cdot,\cdot)$ and Young's inequality, for $\varepsilon>0$,
	\begin{equation}
		\begin{aligned}
			a(u-u_{h},u-v_{h}) & \leq c_a\|u-u_{h}\|_{V}\|u-v_{h}\|_{V} \\
			& \leq\epsilon\|u-u_{h}\|_{V}^2+\frac{c_a^2}{4\epsilon}\|u-v_{h}\|_{V}^2.
		\end{aligned}
	\end{equation}
	For $R(v_h-u)$, since the regularity condition holds,
      we have
	\begin{equation}\label{R(v)}
		R(w)=\int_{\Omega}\left(Luw-f_0w+j_1^0(u;w)\right)\mathrm{d}x+\int_{\Gamma_{S}}\left(\frac{\partial u}{\partial\nu_{L}}w+j_2^0(u;w)\right)\mathrm{d}s.
	\end{equation}
	Using \eqref{partialJz1} and \eqref{partialJz2}, we obtain
	\begin{equation}\label{R(u)estimata}
		\begin{aligned}
			|R(w)| & \leq\left(\|Lu\|_{L^2\left(\Omega\right)}+\|f_0\|_{L^2\left(\Omega\right)}\right)\|w\|_{L^2\left(\Omega\right)}+\|\frac{\partial u}{\partial\nu_{L}}\|_{L^2(\Gamma_{S})}\|w\|_{L^2(\Gamma_{S})}\\
			& \quad +\left(c_{1,0}+c_{1,1}\|u\|_{L^2(\Omega)}\right)\|w\|_{L^2(\Omega)}\\
			& \quad +\left(c_{2,0}+c_{2,1}\|u\|_{L^2(\Gamma_{S})}\right)\|w\|_{L^2(\Gamma_{S})}\\
			& \leq c(u)(||w||_{L^2(\Omega)}+||w||_{L^2(\Gamma_{S})}),
		\end{aligned}
	\end{equation}
	where $c(u)$ depends on $\|u\|_{H^2(\Omega)}$.
	Indeed, by the trace theorem,
	$$\|\frac{\partial u}{\partial\nu_{L}}\|_{L^2(\Gamma_{S})}\leq c\|\nabla u\|_{L^2(\Gamma_{S})}\leq
	c\|u\|_{H^2(\Omega)}.$$
	For $I(u_h,v_h),$ we apply \cref{pro:Jsub} to obtain
	\begin{equation}
		\begin{aligned}
			I(u_{h},v_{h}) & \leq\int_{\Omega}^{}j_1^0(u;u_{h}-u)\mathrm{d}x+\int_{\Omega}^{}j_1^0(u;u-v_{h})\mathrm{d}x\\
			& +\int_{\Omega}^{}j_1^0(u_{h};v_{h}-u)\mathrm{d}x+\int_{\Omega}^{}j_1^0(u_{h};u-u_{h})\mathrm{d}x\\
			& +\int_{\Gamma_{S}}^{}j_2^0(u;u_{h}-u)\mathrm{d}s+\int_{\Gamma_{S}}^{}j_2^0(u;u-v_{h})\mathrm{d}s\\
			& +\int_{\Gamma_{S}}^{}j_2^0(u_{h};v_{h}-u)\mathrm{d}s+\int_{\Gamma_{S}}^{}j_2^0(u_h;u-u_{h})\mathrm{d}s.
		\end{aligned}
	\end{equation}
	
	By \eqref{jioapl}  and \eqref{gamma1u},
	\begin{equation}
		\begin{aligned}
			\int_{\Omega}j_1^0(u;u_h-u)\mathrm{d}x+\int_{\Omega}j_1^0(u_h;u-u_{h})\mathrm{d}x
			&\leq\alpha_1\|u-u_{h}\|_{L^2(\Omega)}^2\\
			&\leq\alpha_1\lambda_L^{-1}\|u-u_{h}\|_{V}^2.
		\end{aligned}
	\end{equation}
	
	By \eqref{partialJz1} 
	\begin{equation}
		\begin{aligned}
			& \int_{\Omega}j_1^0(u;u-v_{h})\mathrm{d}x
			\leq(c_{1,0}+c_{1,1}\|u\|_{L^2(\Omega)})\|u-v_{h}\|_{L^2(\Omega)},\\
			& \int_{\Omega}j_1^0(u_{h};v_{h}-u)\mathrm{d}x
			\leq(c_{1,0}+c_{1,1}\|u_{h}\|_{L^2(\Omega)})\|u-v_{h}\|_{L^2(\Omega)}.
		\end{aligned}
	\end{equation}

	Similarly, by \eqref{j1j2apl}, \eqref{gamma2u} and \eqref{partialJz2},
	\begin{equation}
		\begin{aligned}
			\int_{\Gamma_{S}}j_2^0(u;u_h-u)\mathrm{d}s+\int_{\Gamma_{S}}j_2^0(u_{h};u-u_{h})\mathrm{d}s
			\leq\alpha_2\mu_L^{-1}\|u-u_{h}\|_{V}^2.
		\end{aligned}
	\end{equation}
	\begin{equation}
		\begin{aligned}
			& \int_{\Gamma_{S}}j_2^0(u;u-v_{h})\mathrm{d}s
			\leq(c_{2,0}+c_{2,1}\|u\|_{L^2(\Gamma_{S})})\|u-v_{h}\|_{L^2(\Gamma_{S})},\\
			& \int_{\Gamma_{S}}j_2^0(u_{h};v_{h}-u)\mathrm{d}s
			\leq(c_{2,0}+c_{2,1}\|u_{h}\|_{L^2(\Gamma_{S})})\|u-v_{h}\|_{L^2(\Gamma_{S})}.
		\end{aligned}
	\end{equation}
	
	Combining the above inequalities and using the uniform boundedness of $u_h$ from \cref{uhbounded},
	\begin{equation}
		\begin{aligned}
			I(u_{h},v_{h})\leq{}&\left(\alpha_1\lambda_L^{-1}+\alpha_2\mu_L^{-1}\right)\|u-u_{h}\|_{V}^2 \\
			&+c(u)\left(\|u-v_{h}\|_{_{L^2\left(\Omega\right)}}+\|u-v_{h}\|_{_{L^2\left(\Gamma_{S}\right)}}\right)
		\end{aligned}
	\end{equation}
	Choose
	\begin{equation}\nonumber
		\epsilon=\frac{1}{2}\left(\theta-\alpha_1\lambda_L^{-1}-\alpha_2\mu_L^{-1}\right).
	\end{equation}
	Then
	\begin{equation}\label{erroru-uh1}
		\begin{aligned}
			& (\theta-\alpha_1\lambda_L^{-1}-\alpha_2\mu_L^{-1}-\epsilon)\|u-u_{h}\|_{V}^{} \\
			& \leq c\left(u)[\|u-v_{h}\|\right._{V}^{}+\|u-v_{h}\|_{L^2(\Omega)}^{\frac12}+\|u-v_{h}\|_{L^2(\Gamma_{S})}^{\frac12}].
		\end{aligned}
	\end{equation}
	The desired estimate then follows from the smallness condition \eqref{assa4}.
\end{proof}

\begin{theorem}\label{th:error order}
	Under the conditions of \cref{th:uniqueu}, let $u\in V$ and $u_h\in V_h$ be the solutions of Problem ($P_m$) and Problem ($P_d$), respectively. Assume further that the regularity condition \eqref{uregularity} holds. Then
	\begin{equation}\nonumber\label{oh}
		\|u-u_h\|_V\leq c(u)h,
	\end{equation}
	where $c(u)$ is independent of $h$ and depends on $\|u\|_{H^{2}(\Omega)}$ and $\|u\|_{H^{2}(\Gamma_{S}^{i})}$,  $1\leq i\leq m$.
\end{theorem}
\begin{proof}
	Let $v_h=\Pi_h u$, where $\Pi_h$ denotes the interpolation operator. By the standard interpolation estimates,
	$$\begin{aligned}
		&\|u-v_{h}\|_{V}\leq ch\|u\|_{H^{2}(\Omega)},\\
		&\|u-v_{h}\|_{L^{2}(\Omega)}\leq ch^{2}\|u\|_{H^{2}(\Omega)},\\
		&\|u-v_{h}\|_{L^{2}(\Gamma_{S})}\leq ch^{2}\left(\sum_{i=1}^{m}\|u\|_{H^{2}(\Gamma_{S}^i)}^{2}\right)^{1/2}.
	\end{aligned}$$
	Substituting these estimates into \eqref{eq:cea} yields
	\begin{equation}
		\|u-u_h\|_V\leq c(u)h.
	\end{equation}
\end{proof}

\section{Numerical Examples}\label{sec:numerical example}
In this section, we present numerical examples to verify the theoretical error estimates. The nonsmooth relations are handled using a Lagrange-multiplier-based iterative method \cite{atkin2009theo}. 

In the numerical examples, we take $\Omega=(0,1)\times(0,1),$ $\Gamma_{S}=(0,1)\times\{0\}$, and $\Gamma_{D}=\partial\Omega\backslash\Gamma_{S}.$ We adopt the uniform triangulation as shown in \cref{fig:e1mesh}.

Since the exact solution $u$ is not available, we use the numerical solution at $h = 2^{-9}$ as the ``reference'' solution  $u_{ref}$. The numerical convergence orders with respect to $h$ are then computed from the errors $\|u_{ref}-u_{h}^{}\|_{H^1(\Omega)}$. In what follows, $\lambda_h$ denotes the boundary multiplier associated with $\partial j_2(u_h)$, and $\mu_h$ denotes the interior multiplier associated with $\partial j_1(u_h)$.

For the two nonsmooth potentials $j_1$ and $j_2$, we adopt the following prototype, with different choices of the parameters $a$ and $b$ in the examples:
$$j(t)=\left\{\begin{array}{ll}
	0 & \quad\mathrm{if}\quad t<0,\\
	-e^{-at}+bt+1 & \quad\mathrm{if}\quad t\geq0.
\end{array}\right.$$
The corresponding generalized subdifferential is 
$$\partial j(t) = 
\begin{cases} 
	0 & \text{if } t < 0, \\ 
	[0, a + b] & \text{if } t = 0, \\ 
	a e^{-a t} + b & \text{if } t > 0. 
\end{cases}$$

\begin{example}\label{example1}
	We consider the following problem: 
	\begin{align*}
		-\sum_{i,j=1}^d\frac{\partial}{\partial x_j}\left(a_{ij}\frac{\partial u}{\partial x_i}\right)+a_0u &= f_0+f_1\ \qquad\quad\text{in } \Omega,\\ 
		u &= 0 \qquad\qquad\qquad\text{on } \Gamma_D, \\ 
		-\frac{\partial u}{\partial \nu_L} &\in \partial j_2(u) \qquad\qquad\text{on } \Gamma_{S},
	\end{align*}
	where $-f_{1}\in\partial j_{1}(u)$, $f_0=-40\sin(2\pi x)e^{2y}$,  $a_{ij}=\left[\begin{array}{cc} 2 & 1 \\ 1 & 1 \end{array}\right]$, $a_0 = 0$, and we take $a = b = 1$ for the function $j_1$ and $a = b = 0.5$ for the function $j_2$.
	
	\begin{table}[htbp]
		\centering
		\caption{$H^1$ errors and experimental convergence orders for Example 5.1.}
		\label{tab:example51}
		\vspace{4pt}
		\begin{tabular}{c c c c c c}
			\toprule
			$h$ & $2^{-3}$ & $2^{-4}$ & $2^{-5}$ & $2^{-6}$ & $2^{-7}$ \\
			\midrule
			$\|u_{\rm ref}-u_h\|_{H^1(\Omega)}$ & $2.4346\mathrm{e}\text{+}00$ & $1.2693\mathrm{e}\text{+}00$ & $6.4350\mathrm{e}\text{-}01$ & $3.2560\mathrm{e}\text{-}01$ & $1.5850\mathrm{e}\text{-}01$ \\
			Order & -- & 0.9396 & 0.9801 & 0.9962 & 1.0256 \\
			\bottomrule
		\end{tabular}
	\end{table}
	
	\begin{figure}[H]
		\centering
		\subfigure[uniform triangulation]{
			\raisebox{3.4mm}{\includegraphics[width=0.45\textwidth,trim=12 10 12 10,clip]{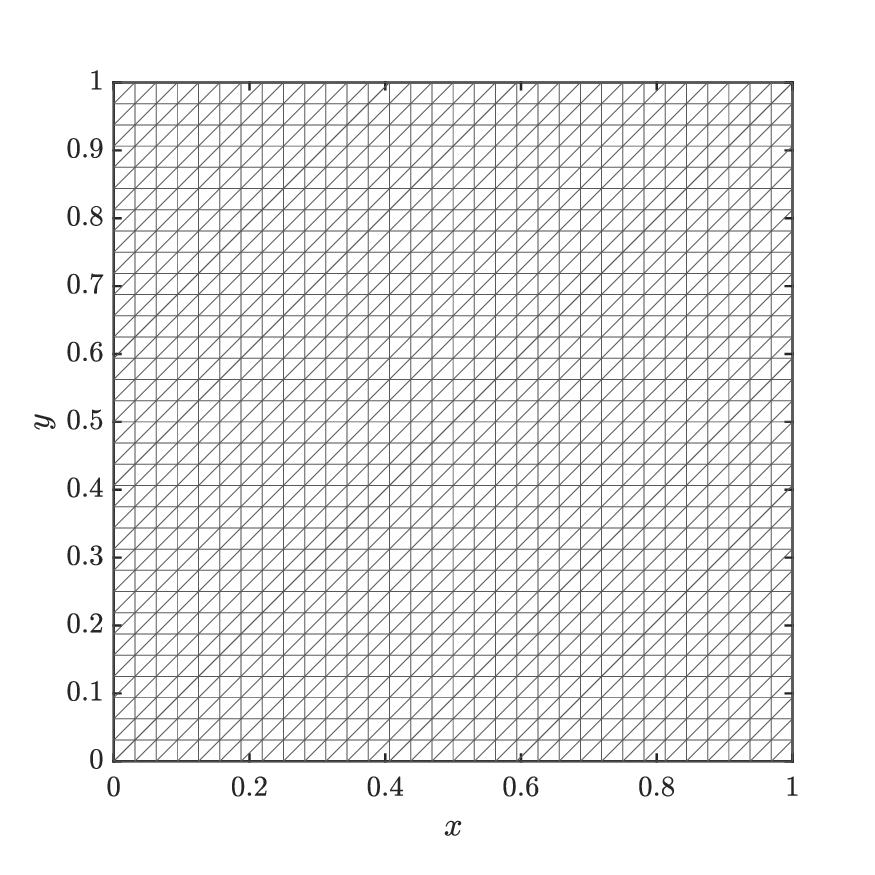}}
			\label{fig:e1mesh}
		}%
		\hfill
		\subfigure[$u_h$ of Example 5.1]{
			\includegraphics[width=0.47\textwidth,trim=34 18 24 18,clip]{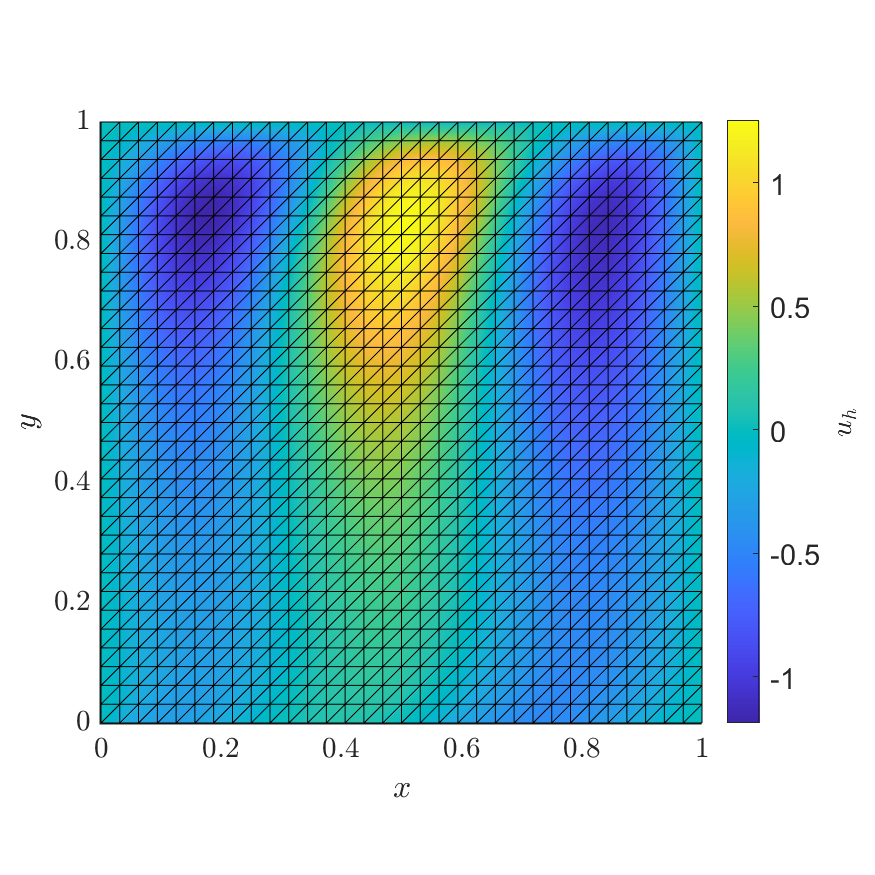}
			\label{fig:e1femnonuh}
		}
		\caption{Triangulation and numerical solution of Example 5.1.}
		\label{fig:meshanduh1}
	\end{figure}
	
	\begin{figure}[H]
		\centering
		\subfigure[$u_h$ and $\lambda_h$]{
			\includegraphics[width=0.45\textwidth]{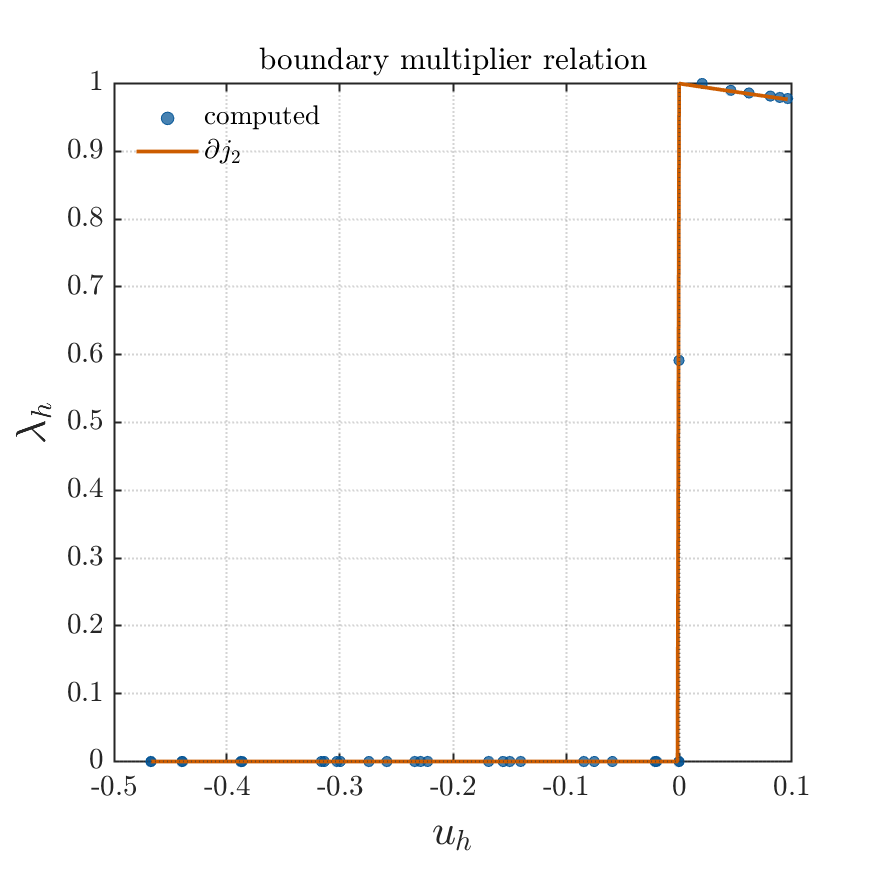}
			\label{fig:e1lambdauh}
		}%
		\hfill
		\subfigure[$u_h$ and $\mu_h$]{
			\includegraphics[width=0.45\textwidth]{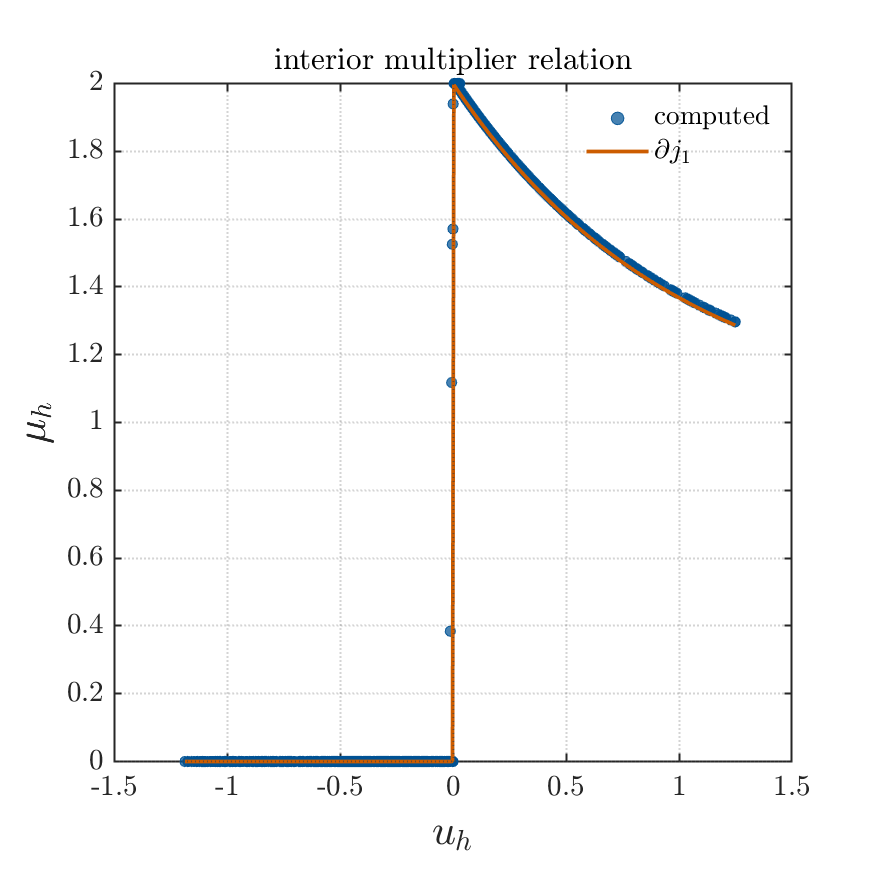}
			\label{fig:e1muuh}
		}
		
		\caption{Multiplier relations for Example 5.1.}
		\label{fig:uhlme1}
	\end{figure}
	
\end{example}

\cref{tab:example51} reports the $H^1$ error of the numerical solution and the corresponding convergence order. We observe that the convergence orders in the $H^1$ norm agree well with the theoretical results presented in \cref{th:error order}. 
\cref{fig:e1femnonuh} shows the numerical solution for $h=2^{-5}$. The observed directional behavior reflects the effect of the non-isotropic diffusion tensor. \cref{fig:uhlme1} compares the computed multipliers with the graphs of the generalized subdifferentials associated with $j_1$ and $j_2$. The numerical results are consistent with the prescribed nonsmooth relations.

To further test the numerical method, \cref{example2} considers a coefficient setting with stronger non-isotropy, where the two diagonal coefficients differ by a factor of 10. In addition, the off-diagonal coefficient contains the function $xy$, which introduces heterogeneity into the diffusion tensor. This example therefore provides a useful test for the performance of the numerical method in a more complex coefficient setting.

\begin{example}\label{example2}
	In this example, we use the same setting as in Example 5.1 except that
	\[
	(a_{ij})=
	\begin{pmatrix}
		1 & xy\\
		xy & 10
	\end{pmatrix},
	\qquad a_0=1,
	\]
	and
	\[
	\begin{aligned}
		f_0 ={}& \big(12\pi^2\sin(2\pi y)+\sin(2\pi y)-2\pi y\cos(2\pi y)\big)\sin(2\pi x)\\
		&-\big(2\pi x\sin(2\pi y)+8\pi^2xy\cos(2\pi y)\big)\cos(2\pi x).
	\end{aligned}
	\] 
	\begin{table}[htbp]
		\centering
		\caption{$H^1$ errors and experimental convergence orders for Example 5.2.}
		\label{tab:example52}
		\vspace{4pt}
		\begin{tabular}{c c c c c c}
			\toprule
			$h$ & $2^{-3}$ & $2^{-4}$ & $2^{-5}$ & $2^{-6}$ & $2^{-7}$ \\
			\midrule
			$\|u_{\rm ref}-u_h\|_{H^1(\Omega)}$ & $6.9671\mathrm{e}\text{-}01$ & $3.5778\mathrm{e}\text{-}01$ & $1.8010\mathrm{e}\text{-}01$ & $8.9772\mathrm{e}\text{-}02$ & $4.3836\mathrm{e}\text{-}02$ \\
			Order & -- & 0.9615 & 0.9902 & 1.0045 & 1.0342 \\
			\bottomrule
		\end{tabular}
	\end{table}
	
	\begin{figure}[H]
		\centering
		\subfigure[$u_h$ and $\lambda_h$]{
			\includegraphics[width=0.45\textwidth]{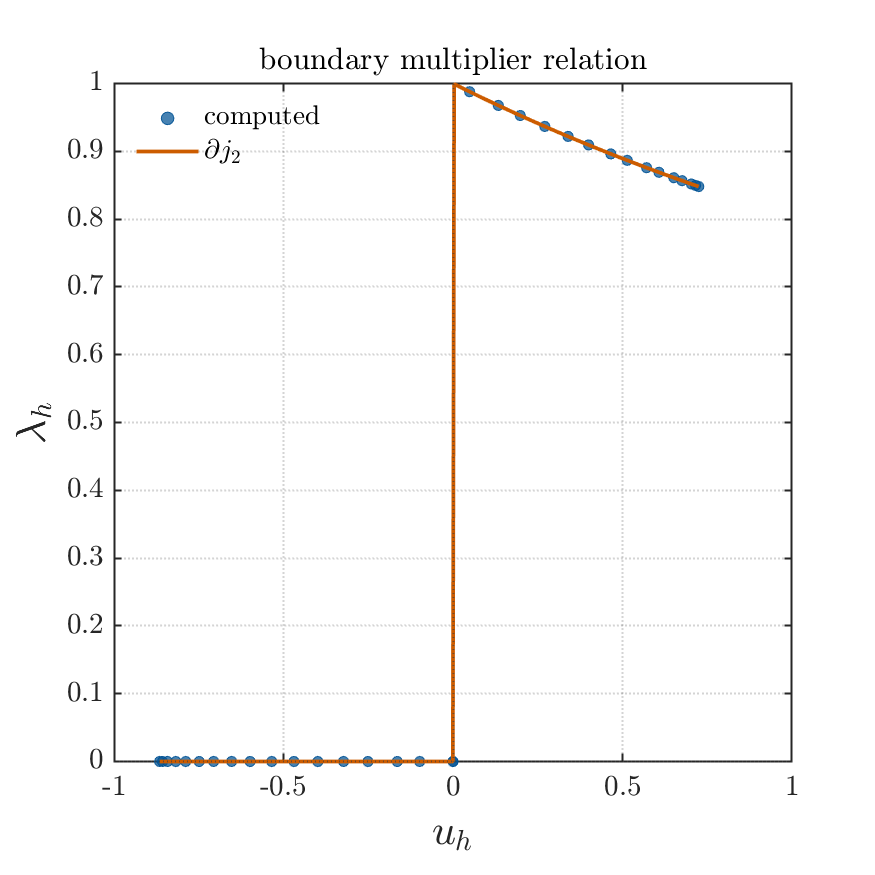}
			\label{fig:e2lambdauh}
		}%
		\hfill
		\subfigure[$u_h$ and $\mu_h$]{
			\includegraphics[width=0.45\textwidth]{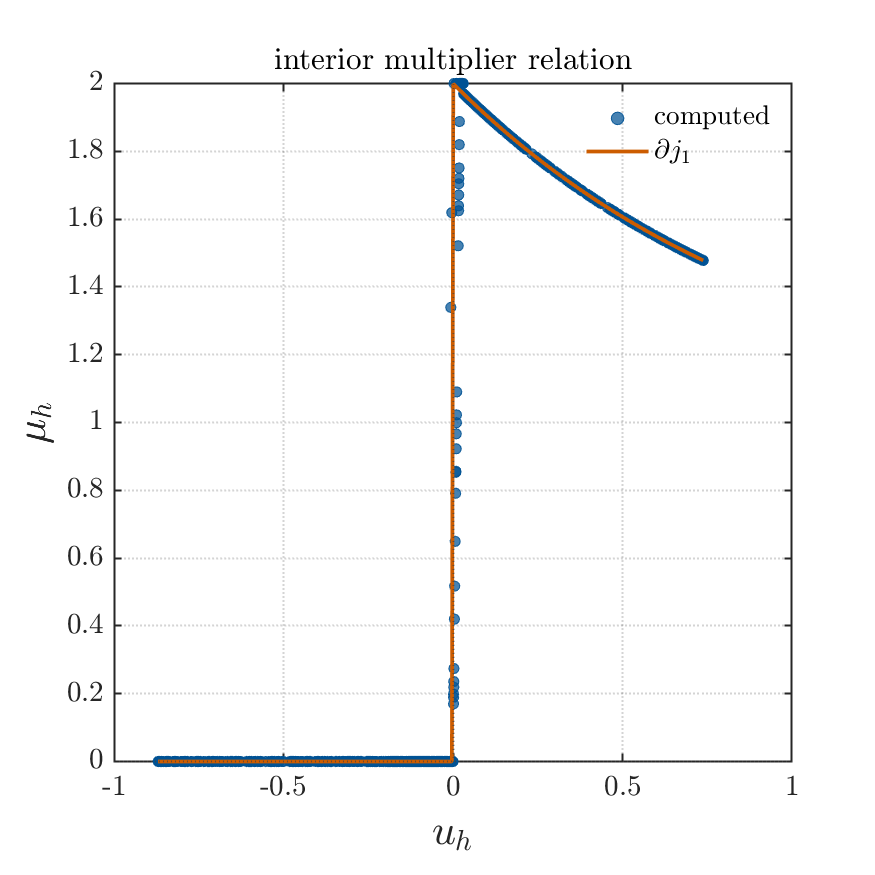}
			\label{fig:e2muuh}
		}
		\caption{Multiplier relations for Example 5.2.}
		\label{fig:uhlme2}
	\end{figure}

\end{example}

As in Example~\ref{example1}, the experimental convergence orders in the $H^1$ norm agree well with the theoretical estimate, as shown in \cref{tab:example52}. 
\cref{fig:uhlme2} further compares the computed multipliers with the corresponding generalized subdifferential graphs. The numerical results are consistent with the prescribed nonsmooth relations in the non-isotropic and heterogeneous case.

\nocite{*}
\bibliographystyle{abbrv}
\bibliography{references}

\end{document}